\documentclass[psamsfonts]{amsart}
\usepackage{amssymb,amsmath,amsthm}
\usepackage{enumitem,multirow,float}
\newtheorem{thm}{Theorem}[section]
\newtheorem{lem}[thm]{Lemma}
\newtheorem{cor}[thm]{Corollary}

\newcommand{\lf}{\left\lfloor}
\newcommand{\rf}{\right\rfloor}

\makeatletter
\newcommand{\addresseshere}{%
  \enddoc@text\let\enddoc@text\relax
}
\makeatother

\begin{document}

\title{A lower bound on the number of rough numbers}
\author{J.Z. Schroeder}
\address{Mosaic Center Radstock \\ Kej Bratstvo Edinstvo 45 \\ 1230 Gostivar \\ Macedonia}
\email{jzschroeder@gmail.com}
\keywords{Rough number, prime number}
\subjclass[2000]{Primary 11N25, Secondary 11A41}
\date{\today}

\begin{abstract}
Conceptually, a {\em rough number} is a positive integer with no small prime factors. Formally, for real numbers $x$ and $y$, let $\Phi(x,y)$ denote the number of positive integers at most $x$ with no prime factors less than $y$. In this paper we establish the lower bound $\Phi(n,p)\geq \lfloor 2n/p \rfloor +1$ when $p\geq 11$ is prime and $n\geq 2p$.
\end{abstract}

\maketitle


\section{Introduction}\label{sec:intro}

As defined in \cite{bib:Finch}, a {\em $y$-rough number} is an integer whose prime factors are all at least $y$. For real numbers $x$ and $y$, the function $\Phi(x,y)$ counts the number of $y$-rough numbers less than or equal to $x$. This function was studied by Buchstab, who showed in \cite{bib:Buchstab} that for any fixed $u>1$, 
\[\Phi(x,x^{1/u})\sim \omega(u)\frac{x}{\log x^{1/u}}, \hspace{12pt} x\rightarrow \infty,\] where $\omega(u)$ is the unique continuous function $\omega:[1,\infty]\rightarrow (0,\infty)$ satisfying \[\begin{array}{ll}
\omega(u)=\frac{1}{u}, & 1\leq u\leq 2, \\
\\
\frac{d}{du}(u\omega(u))=\omega(u-1), & u\geq 2. \\
\end{array}\]

The current study of $\Phi(x,y)$ is motivated by a graph labeling algorithm used in \cite{bib:mainpaper}. For a graph $G$ with $n$ vertices, a {\em prime labeling} is a bijection $f:V(G)\rightarrow \{1,2,...,n\}$ such that $\gcd(f(v),f(w))=1$ for any edge $vw$ in $G$. We seek a prime labeling of a bipartite graph $G=G[A,B]$ with $|A|=|B|=n/2$. We begin by placing the even multiples of 3 at most $n$ on vertices in $A$ and the odd multiples of 3 at most $n$ on vertices in $B$ such that no two vertices labeled with a multiple of 3 are adjacent. Starting with $p=5$ and continuing with all odd primes $p<n$, we place the unused even multiples of $p$ at most $n$ on unlabeled vertices in $A$ and the unused odd multiples of $p$ at most $n$ on unlabeled vertices in $B$ such that no two vertices labeled with a multiple of $p$ are adjacent. If this process can be completed, then $G$ has a prime labeling. At any step of the process, before assigning the unused multiples of a prime $p$, the number of unlabeled vertices in $B$ is given by $\Phi(n,p)$; the goal of this paper is to prove the following lower bound.

\begin{thm}\label{thm:mainthm}
Suppose $p\geq 11$ is prime. Then for all $n\geq 2p$, 
\[\Phi(n,p) \geq \lf\frac{2n}{p}\rf+1.\]
\end{thm}

The proof of this theorem breaks into four cases depending on the value of $n$; these cases are covered by Lemmas \ref{lem:phiboundcase1}-\ref{lem:phiboundcase3} and Lemma \ref{lem:phiboundcase4} in Section \ref{sec:proof}. For a real number $x$, let $\pi(x)$ denote the number of prime numbers less than or equal to $x$; we will need the following results on the distribution of prime numbers.

\begin{lem}[Corollary 1 in \cite{bib:Rosser}]\label{lem:pibound}
For $x>1$, \[\pi(x)<\frac{1.25506x}{\ln x},\] and for $x\geq 17$, \[\pi(x)>\frac{x}{\ln x}.\]
\end{lem}

\begin{lem}[Theorem, page 180 in \cite{bib:Nagura}]\label{lem:Bertrandtype}
Let $x\geq 25$. There is at least one prime number in the interval $\left(x,6x/5\right)$.
\end{lem}

\section{Proof of Theorem \ref{thm:mainthm}}\label{sec:proof}

For any $n$ and $p$, set $z=2n/p$; then $\left\lfloor 2n/p \right\rfloor+1=\lfloor z\rfloor+1$ is a step function that only increments at integer values of $z$. Therefore, in what follows, we can restrict our attention to $\Phi(zp/2,p)$ for integer values of $z$. Namely, several of the proofs in this section require the checking of a finite number of small cases. The author has verified these by hand, and the details are provided in Appendix A.

\begin{lem}\label{lem:phiboundcase1}
If $p\geq 11$ is prime and $2p\leq n\leq 3p$, then $\Phi(n,p) \geq \lf 2n/p \rf+1$.
\end{lem}

\begin{proof}
Let $p\geq 11$ be prime and $2p\leq n\leq 3p$; set $n=pz/2$, where $4\leq z\leq 6$ is a real number. Since we only need to consider integer values of $z$, it suffices to show that $\Phi(2p,p)\geq 5$, $\Phi(5p/2,p)\geq 6$, and $\Phi(3p,p)\geq 7$. 

Let $X_n^p$ denote the set of all positive integers at most $n$ with no prime factors less than $p$, so that $|X_n^p|=\Phi(n,p)$. If $11\leq p\leq 23$, then Table \ref{tab:dataphi1} on page \pageref{tab:dataphi1} provides the set $X_n^p$ and the value of $\Phi(zp/2,p)$ for $z=4,5,6$.

If $p\geq 29$, then the set of integers less than $2p$ that are not divisible by any prime $q<p$ includes $1$ and $p$; moreover, because $p>25$, Lemma \ref{lem:Bertrandtype} implies that there is at least one prime in each interval \[\left(p,\frac{6}{5}p\right),\left(\frac{6}{5}p,\frac{36}{25}p\right),\left(\frac{36}{25}p,\frac{216}{125}p\right).\] Thus, $\Phi(2p,p)\geq 5$. We apply Lemma \ref{lem:Bertrandtype} again to obtain at least one prime in $\left(2p,\frac{12}{5}p\right)\subset \left(2p,\frac{5}{2}p\right)$, so $\Phi(5p/2,p)\geq \Phi(2p,p)+1\geq 6$. Similarly from Lemma \ref{lem:Bertrandtype}, there is at least one prime in $\left(\frac{5}{2}p,3p\right)$, so $\Phi(3p,p)\geq \Phi(5p/2,p)+1\geq 7$, and the proof is complete.
\end{proof}

\begin{lem}\label{lem:phiboundcase2}
If $p\geq 11$ is prime and $3p<n\leq p^2$, then $\Phi(n,p) \geq \lf 2n/p\rf+1$.
\end{lem}

\begin{proof}
Assume first that $11\leq p\leq 89$ and set $n=pz/2$, where $6<z\leq 2p$ is a real number. Note that $\Phi(n,p)$ is a nondecreasing function in $n$, so if $\Phi(n,p)\geq 2p+1$, then $\Phi(n_0,p)\geq 2p+1$ for all $n_0\geq n$. Therefore, for each $p$, we simply need to compute $\Phi(pz/2,p)$ starting with $z=7$ and continuing until we obtain a value at least $2p+1$. The necessary values of $\Phi(pz/2,p)$ for $7\leq z\leq 18$ are given in Table \ref{tab:dataphi2a} on page \pageref{tab:dataphi2a}, and the necessary values of $\Phi(pz/2,p)$ for $19\leq z\leq 28$ are given in Table \ref{tab:dataphi2b} on page \pageref{tab:dataphi2b}. No further computations are required for $11\leq p\leq 89$.

Assume now that $p\geq 97$ and again set $n=pz/2$, where $6<z\leq 2p$ is a real number. We first show that \[\frac{p}{\ln p}\geq \frac{4z-4}{z-5.02024}\] for all $z>6$. When $z=6$, $(4z-4)/(z-5.02024)=20/0.97976\leq 20.42$, and for $p\geq 97$, $p/\ln p\geq 20.42$. The derivative of $f(z)=(4z-4)/(z-5.02024)$ is \[f'(z)=\frac{-16.08096}{(z-5.02024)^2}<0.\] Thus, the right hand side decreases as $z$ increases, so \[\frac{p}{\ln p}\geq \frac{4z-4}{z-5.02024}\] for all $z>6$ as well.

From this we derive that 
\begin{equation}\label{eqn:ineq}
\frac{p}{4\ln p}(z-5.02024)\geq z-1.
\end{equation} Since $\Phi(n,p)$ counts 1, $p$, and all primes greater than $p$ and less than or equal to $n$, we have (using Lemma \ref{lem:pibound}) \[\Phi(n,p)\geq \pi(pz/2)-\pi(p)+2\geq \frac{\frac{pz}{2}}{\ln(\frac{pz}{2})}-\frac{1.25506p}{\ln p}+2.\] Since $z\leq 2p$ implies $pz/2\leq p^2$, we have \[\frac{\frac{pz}{2}}{\ln(\frac{pz}{2})}-\frac{1.25506p}{\ln p}+2\geq \frac{\frac{pz}{2}}{\ln p^2}-\frac{1.25506p}{\ln p}+2=\frac{pz-5.02024p}{4\ln p}+2.\] Finally, using  (\ref{eqn:ineq}), we have \[\frac{pz-5.02024p}{4\ln p}+2\geq z-1+2=\frac{2n}{p}+1\geq \lf\frac{2n}{p}\rf+1\] as desired.
\end{proof}

\begin{lem}\label{lem:phiboundcase3}
If $p\geq 11$ is prime and $p^2<n\leq p^d$, where $d=(p-2c)/2\ln p$ and $c=1.25506$, then $\Phi(n,p) \geq \lf 2n/p\rf+1$.
\end{lem}

\begin{proof}
If $p=11$ then $d<2$, so the statement is vacuously true. Now let $p\geq 13$, and set $n=p^\alpha$, where $2<\alpha\leq d=(p-2c)/2\ln p$ and $c=1.25506$. We know \[\frac{\alpha}{p^{\alpha-2}}<2\] for all $\alpha>2$, because the terms are equal for $\alpha=2$ and \[\frac{d}{d\alpha}\left(\frac{\alpha}{p^{\alpha-2}}\right)=\frac{1-\alpha\ln p}{p^{\alpha-2}}<0.\] From this we obtain \[\frac{c\alpha}{p^{\alpha-2}}< 2c\Rightarrow p-\frac{c\alpha}{p^{\alpha-2}}> p-2c=2d\ln p\geq 2\alpha\ln p.\] Multiplying both sides of the inequality $p-c\alpha/p^{\alpha-2}\geq 2\alpha\ln p$ by $p^{\alpha-1}/\alpha\ln p$ yields \[\frac{p^\alpha-c\alpha p}{\alpha\ln p}\geq 2p^{\alpha-1}\geq \lf\frac{2p^\alpha}{p}\rf.\] Again $\Phi(n,p)$ counts 1, $p$, and all primes greater than $p$ and less than or equal to $n$, so we have (using Lemma \ref{lem:pibound}) \[\Phi(p^\alpha,p)\geq \pi(p^\alpha)-\pi(p)+2\geq \frac{p^\alpha}{\alpha\ln p}-\frac{cp}{\ln p}+2=\frac{p^\alpha-c\alpha p}{\alpha\ln p}+2\geq \lf\frac{2p^\alpha}{p}\rf+1.\] This completes the proof.
\end{proof}

The fourth and final case requires an extended argument. Let $p_1=2,p_2=3,p_3=5,...$ denote the sequence of primes, and for $k\geq 1$ set $Q_k=\prod_{i\leq k} p_i$. Moreover, let $\mu(n)$ be the M\"{o}bius function, defined for a positive integer $n$ by \[\mu(n)=\left\{\begin{array}{ll}
1 & n\text{ is square-free with an even number of prime factors;}\\
-1 & n\text{ is square-free with an odd number of prime factors;}\\ 
0 & n\text{ is not square-free.}\\
\end{array}\right.\] Note that if $d|Q_k$, then $\mu(d)=\pm 1$. We prove the following estimate, which is provided with a brief explanation in \cite[Equation 1.1]{bib:DeBruijn}.

\begin{lem}\label{lem:phiform}
Suppose $p$ is prime. Then \[\Phi(p^\alpha,p)\geq p^\alpha\prod_{p_i<p}\left(1-\frac{1}{p_i}\right)-2^{\pi(p)}.\]
\end{lem}

\begin{proof}
If $p=p_{k+1}$, then an exact formula for $\Phi(p^\alpha,p)$ (sometimes called Legendre's formula) is given by \[\Phi(p^\alpha,p)=\sum_{d|Q_k}\mu(d)\lf \frac{p^\alpha}{d}\rf.\] Set \[\phi(p^\alpha,p)=\sum_{d|Q_k}\mu(d)\frac{p^\alpha}{d};\] we now obtain the approximation \[\left|\Phi(p^\alpha,p)-\phi(p^\alpha,p)\right|=\left|\sum_{d|Q_k}\mu(d)\left(\lf\frac{p^\alpha}{d}\rf-\frac{p^\alpha}{d}\right)\right|\leq \sum_{d|Q_k} 1=2^{\pi(p_k)}<2^{\pi(p)}.\] From this we obtain \[\Phi(p^\alpha,p)\geq \phi(p^\alpha,p)-2^{\pi(p)}.\] Finally, note that \[\frac{\phi(p^\alpha,p)}{p^\alpha}=\sum_{d|Q_k}\mu(d)\frac{1}{d}=\prod_{p_i<p}\left(1-\frac{1}{p_i}\right).\] The result follows.
\end{proof}

\begin{lem}\label{lem:power2bound}
Suppose $p\geq 19$, and let $c=1.25506$. Then \[\frac{2^{\frac{cp}{\ln p}}+1}{p^{\frac{p-2c}{2\ln p}-1}}<\frac{11}{8}.\]
\end{lem}

\begin{proof}
Set \[f(p)=\frac{2^{\frac{cp}{\ln p}}+1}{p^{\frac{p-2c}{2\ln p}-1}}.\] Then $f(19)\approx 1.373<11/8$; to establish the result, we will show that $f$ is decreasing. Let \[f_1(p)=\frac{2^{\frac{cp}{\ln p}}}{p^{\frac{p-2c}{2\ln p}-1}}\text{ and }f_2(p)=\frac{1}{p^{\frac{p-2c}{2\ln p}-1}},\] so that $f=f_1+f_2$. Consider first \[\ln f_1(p)=\frac{cp}{\ln p}\ln 2-\left(\frac{p-2c}{2\ln p}-1\right)\ln p=\frac{(c\ln 2)p}{\ln p}-\frac{p-2c}{2}+\ln p.\] Implicit differentiation yields \[f_1'(p)=f_1(p)\left[\frac{c\ln 2}{\ln p}-\frac{c\ln 2}{\ln^2 p}-\frac{1}{2}+\frac{1}{p}\right].\] Since $f_1(p)>0$ for all $p$ and the expression in brackets is easily shown to be negative for all $p\geq 19$, we conclude that $f_1'(p)<0$. Similarly, we consider \[\ln f_2(p)=-\frac{p}{2}+c\] and see that $f_2'(p)=-f_2(p)/2<0$ as well. Thus $f(p)$ is decreasing, and $f(p)<11/8$ for all $p\geq 19$ as desired.
\end{proof}

For Lemmas \ref{lem:eulerconstant}-\ref{lem:phiboundcase4}, let \[\gamma=\lim_{m\rightarrow \infty} \left(-\ln m+\sum_{k=1}^{m}\frac{1}{k}\right)\approx 0.577\] be the Euler-Mascheroni constant. Note that 2971 and 2999 are consecutive primes.

\begin{lem}\label{lem:eulerconstant}
For a fixed $p$ prime, set $d=(p-2c)/2\ln p$, where $c=1.25506$, and assume $\alpha\geq d$. If $19\leq p\leq 2971$, then \[\frac{p}{e^{\gamma}\ln p}\left(1-\frac{0.2}{\ln^2 p}\right)-0.0058p-\frac{2^{\frac{cp}{\ln p}}}{p^{\alpha-1}}\geq 2+\frac{1}{p^{\alpha-1}},\] and if $p\geq 2999$, then \[\frac{p}{e^{\gamma}\ln p}\left(1-\frac{0.2}{\ln^2 p}\right)-\frac{2^{\frac{cp}{\ln p}}}{p^{\alpha-1}}\geq 2+\frac{1}{p^{\alpha-1}}.\]
\end{lem}

\begin{proof}
Assume first that $19\leq p\leq 2971$, and let $c=1.25506$ and $d=(p-2c)/2\ln p$. Since $\alpha\geq d$, it suffices to show that \[\frac{p}{e^{\gamma}\ln p}\left(1-\frac{0.2}{\ln^2 p}\right)-0.0058p-\frac{2^{\frac{cp}{\ln p}}}{p^{d-1}}\geq 2+\frac{1}{p^{d-1}}.\] Set \[f(p)=\frac{p}{e^{\gamma}\ln p}\left(1-\frac{0.2}{\ln^2 p}\right)-0.0058p-\frac{2^{\frac{cp}{\ln p}}}{p^{d-1}}\] and \[g(p)=2+\frac{1}{p^{d-1}}.\] It is evident that $d>2$ and, hence, $g(p)<2.01$ for all $p\geq 19$, so it suffices to show that $f(p)>2.01$. Tables \ref{tab:dataDusart1} and \ref{tab:dataDusart2} on pages \pageref{tab:dataDusart1} and \pageref{tab:dataDusart2} provide the value of $f(p)$ rounded to two decimal places; clearly, $f(p)>2.01>g(p)$ for all prime $19\leq p\leq 2971$.

Assume now that $p\geq 2999$. Since $\alpha\geq d$, it suffices to show that \[\frac{p}{e^{\gamma}\ln p}\left(1-\frac{0.2}{\ln^2 p}\right)-\frac{2^{\frac{cp}{\ln p}}}{p^{d-1}}\geq 2+\frac{1}{p^{d-1}}.\] Moreover, Lemma \ref{lem:power2bound} implies that it is sufficient to show \[g(p)=\frac{p}{e^{\gamma}\ln p}\left(1-\frac{0.2}{\ln^2 p}\right)\geq \frac{27}{8}.\] But this follows immediately from the fact that $g(2999)\approx 209.7\geq 27/8$ and $g(p)$ is clearly increasing.
\end{proof}

\begin{lem}[Theorem 6.12 in \cite{bib:Dusart}]\label{lem:Dusart}
If $x\geq 2973$, then \[\prod_{p_i\leq x}\left(1-\frac{1}{p_i}\right)>\frac{1}{e^{\gamma}\ln x}\left(1-\frac{0.2}{\ln^2 x}\right).\]
\end{lem}

\begin{cor}\label{cor:Dusartcor}
If $19\leq p\leq 2971$ is prime, then \[\prod_{p_i<p}\left(1-\frac{1}{p_i}\right)>\frac{1}{e^{\gamma}\ln p}\left(1-\frac{0.2}{\ln^2 p}\right)-0.0058,\] and if $p\geq 2999$ is prime, then \[\prod_{p_i<p}\left(1-\frac{1}{p_i}\right)>\frac{1}{e^{\gamma}\ln p}\left(1-\frac{0.2}{\ln^2 p}\right).\]
\end{cor}

\begin{proof}
Assume first that $19\leq p\leq 2971$, and set \[f(p)=\prod_{p_i<p}\left(1-\frac{1}{p_i}\right)\] and \[g(p)=\frac{1}{e^{\gamma}\ln p}\left(1-\frac{0.2}{\ln^2 p}\right)-0.0058.\] Tables \ref{tab:dataDusartCor1}-\ref{tab:dataDusartCor4} on pages \pageref{tab:dataDusartCor1}-\pageref{tab:dataDusartCor4} provide the values of $f(p)$ and $g(p)$ rounded to three decimal places; clearly, $f(p)>g(p)$ for all prime $19\leq p\leq 2971$.

The inequality for $p\geq 2999$ follows immediately from Lemma \ref{lem:Dusart}.
\end{proof}

We are now ready to complete the proof of Theorem \ref{thm:mainthm}. In what follows, we make repeated use of the inequality $\lf x/y\rf \geq x/y-1$.

\begin{lem}\label{lem:phiboundcase4}
If $p\geq 11$ is prime and $n>p^d$, where $d=(p-2c)/2\ln p$ and $c=1.25506$, then $\Phi(n,p) \geq \lf 2n/p\rf+1$.
\end{lem}

\begin{proof}
First, suppose $p=11$ and assume $11^d\approx 69.8 < n<150$. Recall that for a fixed $p$, $\Phi(n,p)$ is increasing in $n$. The maximum value of $\left\lfloor 2n/11 \right\rfloor+1$ for $11^d\approx 69.8 < n<150$ is $\left\lfloor 2(149)/11 \right\rfloor+1=28$, so we simply need to verify the inequality starting at $n=70$ and continuing until we obtain a value of $\Phi(n,11)$ at least 28. The required values are given in Table \ref{tab:dataphi4a} on page \pageref{tab:dataphi4a}, where $\Phi$ stands in the place of $\Phi(n,11)$ and $b=\left\lfloor 2n/11 \right\rfloor+1$.

Assume now that $n\geq 150$; we estimate $\Phi(n,11)$ directly: \[\begin{array}{rcl}
\Phi(n,11) & = & n-\lf\frac{n}{2}\rf-\lf\frac{n}{3}\rf-\lf\frac{n}{5}\rf-\lf\frac{n}{7}\rf+\lf\frac{n}{6}\rf+\lf\frac{n}{10}\rf+\lf\frac{n}{14}\rf+\lf\frac{n}{15}\rf \\
\\
& & +\lf\frac{n}{21}\rf+\lf\frac{n}{35}\rf-\lf\frac{n}{30}\rf-\lf\frac{n}{42}\rf-\lf\frac{n}{70}\rf-\lf\frac{n}{105}\rf+\lf\frac{n}{210}\rf \\
\\
& \geq & n-\frac{n}{2}-\frac{n}{3}-\frac{n}{5}-\frac{n}{7}+\left(\frac{n}{6}-1\right)+\left(\frac{n}{10}-1\right)+\left(\frac{n}{14}-1\right)+\left(\frac{n}{15}-1\right) \\
\\
& & +\left(\frac{n}{21}-1\right)+\left(\frac{n}{35}-1\right)-\frac{n}{30}-\frac{n}{42}-\frac{n}{70}-\frac{n}{105}+\left(\frac{n}{210}-1\right) \\
\\
& = & \frac{48n}{210}-7.
\end{array}\] For $n\geq 150$, \[\frac{48n}{210}-7\geq \lf\frac{2n}{11}\rf+1.\]

Suppose now that $p=13$, and assume $13^d\approx 189.6 < n<297$. Again, $\Phi(n,p)$ is increasing in $n$ for $p=13$. The maximum value of $\left\lfloor 2n/13 \right\rfloor+1$ for $13^d\approx 189.6 < n<297$ is $\left\lfloor 2(296)/13 \right\rfloor+1=46$, so we simply need to verify the inequality starting at $n=190$ and continuing until we obtain a value of $\Phi(n,13)$ at least 46. The required values are given in Table \ref{tab:dataphi4b} on page \pageref{tab:dataphi4b}, where $\Phi$ stands in the place of $\Phi(n,13)$ and $b=\left\lfloor 2n/13 \right\rfloor+1$.

Assume now that $n\geq 297$; we use an argument similar to the one above to obtain $\Phi(n,13)\geq 480n/2310-15$. For $n\geq 297$, \[\frac{480n}{2310}-15\geq \lf\frac{2n}{13}\rf+1.\]

Likewise for $p=17$, we obtain $\Phi(n,17)\geq 5760n/30030-31$. For all $n> 17^d\approx 1401$, \[\frac{5760n}{30030}-31\geq \lf\frac{2n}{17}\rf+1.\]

Suppose now that $19\leq p\leq 2971$ and $n>p^d$. Set $n=p^\alpha$, where $\alpha>d$; from Lemma \ref{lem:phiform} we know \[\Phi(p^\alpha,p)\geq p^\alpha\prod_{p_i<p}\left(1-\frac{1}{p_i}\right)-2^{\pi(p)}.\] Using Lemma \ref{lem:pibound} and Corollary \ref{cor:Dusartcor}, this implies \[\begin{array}{rcl}\Phi(p^\alpha,p)& \geq &  p^\alpha\left[\frac{1}{e^{\gamma}\ln p}\left(1-\frac{0.2}{\ln^2 p}\right)-0.0058\right]-2^{\frac{cp}{\ln p}}\\ \\
& = & p^{\alpha-1}\left[\frac{p}{e^{\gamma}\ln p}\left(1-\frac{0.2}{\ln^2 p}\right)-0.0058p-\frac{2^{\frac{cp}{\ln p}}}{p^{\alpha-1}}\right],
\end{array}\] and using Lemma \ref{lem:eulerconstant} yields \[\Phi(p^\alpha,p)\geq p^{\alpha-1}\left(2+\frac{1}{p^{\alpha-1}}\right)=2p^{\alpha-1}+1\geq \lf\frac{2p^\alpha}{p}\rf+1.\]

Finally, suppose that $p\geq 2999$ and $n>p^d$. We again set $n=p^\alpha$, where $\alpha>d$; using an argument similar to the preceding paragraph, we obtain via Lemmas \ref{lem:pibound}, \ref{lem:phiform}, \ref{lem:eulerconstant} and Corollary \ref{cor:Dusartcor} that \[\begin{array}{rcl} 
\Phi(p^\alpha,p) & \geq & {\displaystyle p^\alpha\prod_{p_i<p}\left(1-\frac{1}{p_i}\right)-2^{\pi(p)} }\\ \\
& \geq &  p^\alpha\frac{1}{e^{\gamma}\ln p}\left(1-\frac{0.2}{\ln^2 p}\right)-2^{\frac{cp}{\ln p}}\\ \\
& = & p^{\alpha-1}\left[\frac{p}{e^{\gamma}\ln p}\left(1-\frac{0.2}{\ln^2 p}\right)-\frac{2^{\frac{cp}{\ln p}}}{p^{\alpha-1}}\right] \\ \\
& \geq & p^{\alpha-1}\left(2+\frac{1}{p^{\alpha-1}}\right)=2p^{\alpha-1}+1\geq \lf\frac{2p^\alpha}{p}\rf+1.
\end{array}\] This completes the proof.
\end{proof}


\addresseshere



\newpage

\section*{Appendix A: The tables}

\begin{table}[H]
\centerline{\begin{tabular}{|l|l|l|l|}
\hline
\multicolumn{1}{|c|}{$p$} & \multicolumn{1}{|c|}{$n$} & \multicolumn{1}{|c|}{$\Phi(n,p)$} & \multicolumn{1}{|c|}{$X_n^p$} \\
\hline
\multirow{3}{*}{11} & 22 & 5 & $\{1,11,13,17,19\}$ \\
& 27.5 & 6 & $\{1,11,13,17,19,23\}$ \\
& 33 & 8 & $\{1,11,13,17,19,23,29,31\}$ \\
\hline
\multirow{3}{*}{13} & 26 & 5 & $\{1,13,17,19,23\}$ \\
& 32.5 & 7 & $\{1,13,17,19,23,29,31\}$ \\
& 39 & 8 & $\{1,13,17,19,23,29,31,37\}$ \\
\hline
\multirow{3}{*}{17} & 34 & 6 & $\{1,17,19,23,29,31\}$ \\
& 42.5 & 8 & $\{1,17,19,23,29,31,37,41\}$ \\
& 51 & 10 & $\{1,17,19,23,29,31,37,41,43,47\}$ \\
\hline
\multirow{3}{*}{19} & 38 & 6 & $\{1,19,23,29,31,37\}$ \\
& 47.5 & 9 & $\{1,19,23,29,31,37,41,43,47\}$ \\
& 57 & 10 & $\{1,19,23,29,31,37,41,43,47,53\}$ \\
\hline
\multirow{3}{*}{23} & 46 & 7 & $\{1,23,29,31,37,41,43\}$ \\
& 57.5 & 9 & $\{1,23,29,31,37,41,43,47,53\}$ \\
& 69 & 12 & $\{1,23,29,31,37,41,43,47,53,59,61,67\}$ \\
\hline
\end{tabular}}
\caption{Data for the proof of Lemma \ref{lem:phiboundcase1}.}\label{tab:dataphi1}
\end{table}

\begin{table}[H]
\centerline{\begin{tabular}{l|cccccccccccc}
\multicolumn{1}{r|}{$z=$}  & 7 & 8 & 9 & 10 & 11 & 12 & 13 & 14 & 15 & 16 & 17 & 18 \\
\hline
$\Phi(11z/2),11)$ & 9 & 11 & 12 & 13 & 14 & 15 & 17 & 18 & 19 & 20 & 21 & 22 \\
$\Phi(13z/2),13)$ & 10 & 11 & 12 & 14 & 16 & 17 & 19 & 20 & 21 & 23 & 25 & 26 \\
$\Phi(17z/2),17)$ & 12 & 14 & 16 & 18 & 19 & 21 & 24 & 25 & 26 & 27 & 29 & 31 \\
$\Phi(19z/2),19)$ & 12 & 15 & 17 & 18 & 21 & 24 & 24 & 26 & 28 & 30 & 31 & 33 \\
$\Phi(23z/2),23)$ & 15 & 17 & 20 & 23 & 23 & 26 & 28 & 30 & 32 & 35 & 37 & 39 \\
$\Phi(29z/2),29)$ & 18 & 22 & 23 & 26 & 29 & 32 & 34 & 38 & 39 & 42 & 45 & 47 \\
$\Phi(31z/2),31)$ & 19 & 21 & 25 & 27 & 30 & 33 & 37 & 38 & 41 & 44 & 47 & 50 \\
$\Phi(37z/2),37)$ & 21 & 24 & 28 & 32 & 36 & 37 & 42 & 45 & 49 & 52 & 55 & 57 \\
$\Phi(41z/2),41)$ & 23 & 27 & 31 & 35 & 37 & 42 & 45 & 50 & 52 & 55 & 58 & 62 \\
$\Phi(43z/2),43)$ & 23 & 27 & 32 & 35 & 39 & 43 & 47 & 50 & 54 & 56 & 60 & 64 \\
$\Phi(47z/2),47)$ & 25 & 29 & 34 & 38 & 42 & 47 & 49 & 53 & 57 & 61 & 65 & 69 \\
$\Phi(53z/2),53)$ & 28 & 33 & 37 & 42 & 47 & 52 & 54 & 59 & 64 & 68 & 73 & 77 \\
$\Phi(59z/2),59)$ & 31 & 36 & 41 & 47 & 51 & 56 & 61 & 65 & 70 & 76 & 80 & 84 \\
$\Phi(61z/2),61)$ & 31 & 37 & 42 & 46 & 51 & 56 & 61 & 66 & 72 & 77 & 81 & 85 \\
$\Phi(67z/2),67)$ & 34 & 39 & 45 & 50 & 56 & 62 & 67 & 74 & 78 & 82 & 87 & 93 \\
$\Phi(71z/2),71)$ & 35 & 43 & 48 & 53 & 59 & 64 & 71 & 76 & 81 & 85 & 92 & 97 \\
$\Phi(73z/2),73)$ & 35 & 42 & 47 & 53 & 60 & 65 & 72 & 78 & 82 & 87 & 95 & 100 \\
$\Phi(79z/2),79)$ & 38 & 45 & 51 & 57 & 64 & 71 & 77 & 81 & 87 & 95 & 101 & 107 \\
$\Phi(83z/2),83)$ & 40 & 46 & 53 & 59 & 66 & 73 & 78 & 85 & 93 & 100 & 105 & 111 \\
$\Phi(89z/2),89)$ & 42 & 49 & 56 & 64 & 71 & 77 & 84 & 92 & 99 & 105 & 111 & 117 \\
\end{tabular}}
\caption{Data for the proof of Lemma \ref{lem:phiboundcase2} for $7\leq z\leq 18$.}\label{tab:dataphi2a}
\end{table}

\begin{table}[H]
\centerline{\begin{tabular}{l|cccccccccc}
\multicolumn{1}{r|}{$z=$}  & 19 & 20 & 21 & 22 & 23 & 24 & 25 & 26 & 27 & 28 \\
\hline
$\Phi(11z/2),11)$ & 24 & \multicolumn{9}{c}{} \\
$\Phi(13z/2),13)$ & 26 & 27 & \multicolumn{8}{c}{} \\
$\Phi(17z/2),17)$ & 32 & 34 & 35 & \multicolumn{7}{c}{} \\
$\Phi(19z/2),19)$ & 35 & 36 & 40 & \multicolumn{7}{c}{} \\
$\Phi(23z/2),23)$ & 40 & 43 & 46 & 47 & \multicolumn{6}{c}{} \\
$\Phi(29z/2),29)$ & 50 & 53 & 54 & 58 & 59 & \multicolumn{5}{c}{} \\
$\Phi(31z/2),31)$ & 53 & 54 & 57 & 59 & 62 & 64 & \multicolumn{4}{c}{} \\
$\Phi(37z/2),37)$ & 60 & 63 & 66 & 69 & 72 & 76 & \multicolumn{4}{c}{} \\
$\Phi(41z/2),41)$ & 66 & 69 & 71 & 76 & 80 & 83 & \multicolumn{4}{c}{} \\
$\Phi(43z/2),43)$ & 67 & 70 & 75 & 79 & 82 & 85 & 87 & \multicolumn{3}{c}{} \\
$\Phi(47z/2),47)$ & 73 & 78 & 81 & 84 & 86 & 90 & 94 & 98 & \multicolumn{2}{c}{} \\
$\Phi(53z/2),53)$ & 82 & 85 & 87 & 92 & 97 & 101 & 107 & \multicolumn{3}{c}{} \\
$\Phi(59z/2),59)$ & 87 & 92 & 99 & 103 & 108 & 111 & 115 & 120 & \multicolumn{2}{c}{} \\
$\Phi(61z/2),61)$ & 90 & 95 & 99 & 105 & 110 & 113 & 119 & 122 & 127 & \\
$\Phi(67z/2),67)$ & 98 & 104 & 109 & 113 & 119 & 122 & 128 & 133 & 137 & \\
$\Phi(71z/2),71)$ & 104 & 109 & 114 & 119 & 123 & 128 & 136 & 139 & 144 & \\
$\Phi(73z/2),73)$ & 106 & 110 & 116 & 120 & 127 & 131 & 137 & 142 & 147 & \\
$\Phi(79z/2),79)$ & 112 & 118 & 125 & 130 & 135 & 141 & 146 & 152 & 159 & \\
$\Phi(83z/2),83)$ & 117 & 124 & 129 & 135 & 141 & 146 & 153 & 159 & 166 & 170 \\
$\Phi(89z/2),89)$ & 124 & 132 & 136 & 143 & 150 & 157 & 164 & 169 & 175 & 181 \\
\end{tabular}}
\caption{Data for the proof of Lemma \ref{lem:phiboundcase2} for $19\leq z\leq 28$.}\label{tab:dataphi2b}
\end{table}

\begin{table}[H]
\centerline{\begin{tabular}{ll|ll|ll|ll|ll}
\multicolumn{1}{c}{$p$} & \multicolumn{1}{c|}{$f(p)$} & \multicolumn{1}{c}{$p$} & \multicolumn{1}{c|}{$f(p)$} & \multicolumn{1}{c}{$p$} & \multicolumn{1}{c|}{$f(p)$} & \multicolumn{1}{c}{$p$} & \multicolumn{1}{c|}{$f(p)$} & \multicolumn{1}{c}{$p$} & \multicolumn{1}{c}{$f(p)$} \\
\hline
19 & 2.06 & 233 & 22.49 & 491 & 41.41 & 773 & 60.48 & 1069 & 79.50 \\
23 & 3.42 & 239 & 22.95 & 499 & 41.97 & 787 & 61.40 & 1087 & 80.63 \\
29 & 4.49 & 241 & 23.11 & 503 & 42.25 & 797 & 62.06 & 1091 & 80.89 \\
31 & 4.75 & 251 & 23.88 & 509 & 42.67 & 809 & 62.84 & 1093 & 81.01 \\
37 & 5.44 & 257 & 24.34 & 521 & 43.50 & 811 & 62.97 & 1097 & 81.26 \\
41 & 5.87 & 263 & 24.80 & 523 & 43.64 & 821 & 63.62 & 1103 & 81.64 \\
43 & 6.08 & 269 & 25.26 & 541 & 44.88 & 823 & 63.76 & 1109 & 82.02 \\
47 & 6.49 & 271 & 25.42 & 547 & 45.30 & 827 & 64.02 & 1117 & 82.52 \\
53 & 7.09 & 277 & 25.87 & 557 & 45.98 & 829 & 64.15 & 1123 & 82.89 \\
59 & 7.68 & 281 & 26.18 & 563 & 46.40 & 839 & 64.80 & 1129 & 83.27 \\
61 & 7.88 & 283 & 26.33 & 569 & 46.81 & 853 & 65.71 & 1151 & 84.64 \\
67 & 8.46 & 293 & 27.08 & 571 & 46.95 & 857 & 65.97 & 1153 & 84.77 \\
71 & 8.84 & 307 & 28.13 & 577 & 47.36 & 859 & 66.09 & 1163 & 85.39 \\
73 & 9.03 & 311 & 28.43 & 587 & 48.04 & 863 & 66.35 & 1171 & 85.89 \\
79 & 9.59 & 313 & 28.58 & 593 & 48.45 & 877 & 67.26 & 1181 & 86.51 \\
83 & 9.96 & 317 & 28.88 & 599 & 48.86 & 881 & 67.52 & 1187 & 86.88 \\
89 & 10.51 & 331 & 29.92 & 601 & 48.99 & 883 & 67.65 & 1193 & 87.25 \\
97 & 11.23 & 337 & 30.36 & 607 & 49.40 & 887 & 67.91 & 1201 & 87.75 \\
101 & 11.59 & 347 & 31.10 & 613 & 49.81 & 907 & 69.19 & 1213 & 88.50 \\
103 & 11.76 & 349 & 31.25 & 617 & 50.08 & 911 & 69.45 & 1217 & 88.74 \\
107 & 12.12 & 353 & 31.54 & 619 & 50.21 & 919 & 69.97 & 1223 & 89.11 \\
109 & 12.29 & 359 & 31.98 & 631 & 51.03 & 929 & 70.61 & 1229 & 89.49 \\
113 & 12.65 & 367 & 32.56 & 641 & 51.70 & 937 & 71.12 & 1231 & 89.61 \\
127 & 13.86 & 373 & 33.00 & 643 & 51.84 & 941 & 71.38 & 1237 & 89.98 \\
131 & 14.20 & 379 & 33.44 & 647 & 52.11 & 947 & 71.76 & 1249 & 90.72 \\
137 & 14.71 & 383 & 33.73 & 653 & 52.51 & 953 & 72.14 & 1259 & 91.34 \\
139 & 14.88 & 389 & 34.16 & 659 & 52.91 & 967 & 73.04 & 1277 & 92.45 \\
149 & 15.72 & 397 & 34.74 & 661 & 53.05 & 971 & 73.29 & 1279 & 92.57 \\
151 & 15.89 & 401 & 35.03 & 673 & 53.85 & 977 & 73.68 & 1283 & 92.82 \\
157 & 16.39 & 409 & 35.60 & 677 & 54.12 & 983 & 74.06 & 1289 & 93.19 \\
163 & 16.88 & 419 & 36.32 & 683 & 54.52 & 991 & 74.57 & 1291 & 93.31 \\
167 & 17.21 & 421 & 36.46 & 691 & 55.05 & 997 & 74.95 & 1297 & 93.68 \\
173 & 17.70 & 431 & 37.18 & 701 & 55.72 & 1009 & 75.71 & 1301 & 93.92 \\
179 & 18.19 & 433 & 37.32 & 709 & 56.25 & 1013 & 75.96 & 1303 & 94.05 \\
181 & 18.35 & 439 & 37.74 & 719 & 56.92 & 1019 & 76.34 & 1307 & 94.29 \\
191 & 19.16 & 443 & 38.03 & 727 & 57.45 & 1021 & 76.47 & 1319 & 95.03 \\
193 & 19.32 & 449 & 38.45 & 733 & 57.84 & 1031 & 77.10 & 1321 & 95.15 \\
197 & 19.64 & 457 & 39.02 & 739 & 58.24 & 1033 & 77.23 & 1327 & 95.52 \\
199 & 19.80 & 461 & 39.30 & 743 & 58.51 & 1039 & 77.61 & 1361 & 97.60 \\
211 & 20.76 & 463 & 39.44 & 751 & 59.03 & 1049 & 78.24 & 1367 & 97.96 \\
223 & 21.70 & 467 & 39.73 & 757 & 59.43 & 1051 & 78.37 & 1373 & 98.33 \\
227 & 22.02 & 479 & 40.57 & 761 & 59.69 & 1061 & 79.00 & 1381 & 98.82 \\
229 & 22.17 & 487 & 41.13 & 769 & 60.22 & 1063 & 79.12 & 1399 & 99.91 \\
\end{tabular}}
\caption{Data for the proof of Lemma \ref{lem:eulerconstant} for $19\leq p\leq 1399$.}\label{tab:dataDusart1}
\end{table}

\begin{table}[H]
\centerline{\begin{tabular}{ll|ll|ll|ll|ll}
\multicolumn{1}{c}{$p$} & \multicolumn{1}{c|}{$f(p)$} & \multicolumn{1}{c}{$p$} & \multicolumn{1}{c|}{$f(p)$} & \multicolumn{1}{c}{$p$} & \multicolumn{1}{c|}{$f(p)$} & \multicolumn{1}{c}{$p$} & \multicolumn{1}{c|}{$f(p)$} & \multicolumn{1}{c}{$p$} & \multicolumn{1}{c}{$f(p)$} \\
\hline
1409 & 100.52 & 1699 & 117.94 & 2039 & 137.89 & 2377 & 157.33 & 2711 & 176.21 \\
1423 & 101.37 & 1709 & 118.53 & 2053 & 138.70 & 2381 & 157.56 & 2713 & 176.32 \\
1427 & 101.61 & 1721 & 119.24 & 2063 & 139.28 & 2383 & 157.67 & 2719 & 176.66 \\
1429 & 101.73 & 1723 & 119.36 & 2069 & 139.63 & 2389 & 158.01 & 2729 & 177.22 \\
1433 & 101.98 & 1733 & 119.95 & 2081 & 140.33 & 2393 & 158.24 & 2731 & 177.33 \\
1439 & 102.34 & 1741 & 120.43 & 2083 & 140.44 & 2399 & 158.58 & 2741 & 177.89 \\
1447 & 102.83 & 1747 & 120.78 & 2087 & 140.67 & 2411 & 159.26 & 2749 & 178.34 \\
1451 & 103.07 & 1753 & 121.14 & 2089 & 140.79 & 2417 & 159.60 & 2753 & 178.56 \\
1453 & 103.19 & 1759 & 121.49 & 2099 & 141.37 & 2423 & 159.95 & 2767 & 179.35 \\
1459 & 103.55 & 1777 & 122.55 & 2111 & 142.06 & 2437 & 160.74 & 2777 & 179.91 \\
1471 & 104.28 & 1783 & 122.91 & 2113 & 142.18 & 2441 & 160.97 & 2789 & 180.58 \\
1481 & 104.88 & 1787 & 123.14 & 2129 & 143.10 & 2447 & 161.31 & 2791 & 180.69 \\
1483 & 105.00 & 1789 & 123.26 & 2131 & 143.22 & 2459 & 161.99 & 2797 & 181.02 \\
1487 & 105.24 & 1801 & 123.97 & 2137 & 143.56 & 2467 & 162.45 & 2801 & 181.25 \\
1489 & 105.37 & 1811 & 124.56 & 2141 & 143.79 & 2473 & 162.79 & 2803 & 181.36 \\
1493 & 105.61 & 1823 & 125.27 & 2143 & 143.91 & 2477 & 163.01 & 2819 & 182.25 \\
1499 & 105.97 & 1831 & 125.74 & 2153 & 144.49 & 2503 & 164.49 & 2833 & 183.04 \\
1511 & 106.69 & 1847 & 126.68 & 2161 & 144.95 & 2521 & 165.50 & 2837 & 183.26 \\
1523 & 107.41 & 1861 & 127.50 & 2179 & 145.99 & 2531 & 166.07 & 2843 & 183.59 \\
1531 & 107.90 & 1867 & 127.85 & 2203 & 147.37 & 2539 & 166.52 & 2851 & 184.04 \\
1543 & 108.62 & 1871 & 128.09 & 2207 & 147.60 & 2543 & 166.75 & 2857 & 184.37 \\
1549 & 108.98 & 1873 & 128.20 & 2213 & 147.94 & 2549 & 167.09 & 2861 & 184.60 \\
1553 & 109.22 & 1877 & 128.44 & 2221 & 148.40 & 2551 & 167.20 & 2879 & 185.60 \\
1559 & 109.58 & 1879 & 128.56 & 2237 & 149.32 & 2557 & 167.54 & 2887 & 186.05 \\
1567 & 110.06 & 1889 & 129.14 & 2239 & 149.44 & 2579 & 168.78 & 2897 & 186.60 \\
1571 & 110.30 & 1901 & 129.84 & 2243 & 149.66 & 2591 & 169.46 & 2903 & 186.94 \\
1579 & 110.78 & 1907 & 130.20 & 2251 & 150.12 & 2593 & 169.57 & 2909 & 187.27 \\
1583 & 111.02 & 1913 & 130.55 & 2267 & 151.04 & 2609 & 170.47 & 2917 & 187.72 \\
1597 & 111.86 & 1931 & 131.60 & 2269 & 151.16 & 2617 & 170.92 & 2927 & 188.27 \\
1601 & 112.10 & 1933 & 131.72 & 2273 & 151.38 & 2621 & 171.15 & 2939 & 188.94 \\
1607 & 112.45 & 1949 & 132.65 & 2281 & 151.84 & 2633 & 171.83 & 2953 & 189.72 \\
1609 & 112.57 & 1951 & 132.77 & 2287 & 152.19 & 2647 & 172.61 & 2957 & 189.94 \\
1613 & 112.81 & 1973 & 134.05 & 2293 & 152.53 & 2657 & 173.18 & 2963 & 190.27 \\
1619 & 113.17 & 1979 & 134.40 & 2297 & 152.76 & 2659 & 173.29 & 2969 & 190.60 \\
1621 & 113.29 & 1987 & 134.87 & 2309 & 153.45 & 2663 & 173.51 & 2971 & 190.72 \\
1627 & 113.65 & 1993 & 135.22 & 2311 & 153.56 & 2671 & 173.96 & &
\end{tabular}}
\caption{Data for the proof of Lemma \ref{lem:eulerconstant} for $1409\leq p\leq 2971$.}\label{tab:dataDusart2}
\end{table}

\begin{table}[H]
\centerline{\begin{tabular}{lll|lll|lll}
\multicolumn{1}{c}{$p$} & \multicolumn{1}{c}{$f(p)$} & \multicolumn{1}{c|}{$g(p)$} & \multicolumn{1}{c}{$p$} & \multicolumn{1}{c}{$f(p)$} & \multicolumn{1}{c|}{$g(p)$} & \multicolumn{1}{c}{$p$} & \multicolumn{1}{c}{$f(p)$} & \multicolumn{1}{c}{$g(p)$} \\
\hline
19 & 0.181 & 0.180 & 233 & 0.102 & 0.097 & 491 & 0.090 & 0.084 \\
23 & 0.171 & 0.170 & 239 & 0.102 & 0.096 & 499 & 0.090 & 0.084 \\
29 & 0.164 & 0.158 & 241 & 0.101 & 0.096 & 503 & 0.090 & 0.084 \\
31 & 0.158 & 0.155 & 251 & 0.101 & 0.095 & 509 & 0.089 & 0.084 \\
37 & 0.153 & 0.147 & 257 & 0.100 & 0.095 & 521 & 0.089 & 0.083 \\
41 & 0.149 & 0.143 & 263 & 0.100 & 0.094 & 523 & 0.089 & 0.083 \\
43 & 0.145 & 0.141 & 269 & 0.100 & 0.094 & 541 & 0.089 & 0.083 \\
47 & 0.142 & 0.138 & 271 & 0.099 & 0.094 & 547 & 0.089 & 0.083 \\
53 & 0.139 & 0.134 & 277 & 0.099 & 0.093 & 557 & 0.089 & 0.083 \\
59 & 0.136 & 0.130 & 281 & 0.098 & 0.093 & 563 & 0.088 & 0.082 \\
61 & 0.134 & 0.129 & 283 & 0.098 & 0.093 & 569 & 0.088 & 0.082 \\
67 & 0.132 & 0.126 & 293 & 0.098 & 0.092 & 571 & 0.088 & 0.082 \\
71 & 0.130 & 0.124 & 307 & 0.097 & 0.092 & 577 & 0.088 & 0.082 \\
73 & 0.128 & 0.124 & 311 & 0.097 & 0.091 & 587 & 0.088 & 0.082 \\
79 & 0.126 & 0.121 & 313 & 0.097 & 0.091 & 593 & 0.088 & 0.082 \\
83 & 0.124 & 0.120 & 317 & 0.097 & 0.091 & 599 & 0.088 & 0.082 \\
89 & 0.123 & 0.118 & 331 & 0.096 & 0.090 & 601 & 0.087 & 0.082 \\
97 & 0.122 & 0.116 & 337 & 0.096 & 0.090 & 607 & 0.087 & 0.081 \\
101 & 0.120 & 0.115 & 347 & 0.096 & 0.090 & 613 & 0.087 & 0.081 \\
103 & 0.119 & 0.114 & 349 & 0.095 & 0.090 & 617 & 0.087 & 0.081 \\
107 & 0.118 & 0.113 & 353 & 0.095 & 0.089 & 619 & 0.087 & 0.081 \\
109 & 0.117 & 0.113 & 359 & 0.095 & 0.089 & 631 & 0.087 & 0.081 \\
113 & 0.116 & 0.112 & 367 & 0.095 & 0.089 & 641 & 0.087 & 0.081 \\
127 & 0.115 & 0.109 & 373 & 0.094 & 0.088 & 643 & 0.086 & 0.081 \\
131 & 0.114 & 0.108 & 379 & 0.094 & 0.088 & 647 & 0.086 & 0.081 \\
137 & 0.113 & 0.107 & 383 & 0.094 & 0.088 & 653 & 0.086 & 0.080 \\
139 & 0.112 & 0.107 & 389 & 0.094 & 0.088 & 659 & 0.086 & 0.080 \\
149 & 0.111 & 0.106 & 397 & 0.093 & 0.088 & 661 & 0.086 & 0.080 \\
151 & 0.111 & 0.105 & 401 & 0.093 & 0.087 & 673 & 0.086 & 0.080 \\
157 & 0.110 & 0.104 & 409 & 0.093 & 0.087 & 677 & 0.086 & 0.080 \\
163 & 0.109 & 0.104 & 419 & 0.093 & 0.087 & 683 & 0.085 & 0.080 \\
167 & 0.109 & 0.103 & 421 & 0.092 & 0.087 & 691 & 0.085 & 0.080 \\
173 & 0.108 & 0.102 & 431 & 0.092 & 0.086 & 701 & 0.085 & 0.079 \\
179 & 0.107 & 0.102 & 433 & 0.092 & 0.086 & 709 & 0.085 & 0.079 \\
181 & 0.107 & 0.101 & 439 & 0.092 & 0.086 & 719 & 0.085 & 0.079 \\
191 & 0.106 & 0.100 & 443 & 0.092 & 0.086 & 727 & 0.085 & 0.079 \\
193 & 0.105 & 0.100 & 449 & 0.091 & 0.086 & 733 & 0.085 & 0.079 \\
197 & 0.105 & 0.100 & 457 & 0.091 & 0.085 & 739 & 0.085 & 0.079 \\
199 & 0.104 & 0.100 & 461 & 0.091 & 0.085 & 743 & 0.085 & 0.079 \\
211 & 0.104 & 0.098 & 463 & 0.091 & 0.085 & 751 & 0.084 & 0.079 \\
223 & 0.103 & 0.097 & 467 & 0.091 & 0.085 & 757 & 0.084 & 0.079 \\
227 & 0.103 & 0.097 & 479 & 0.090 & 0.085 & 761 & 0.084 & 0.078 \\
229 & 0.102 & 0.097 & 487 & 0.090 & 0.084 & 769 & 0.084 & 0.078 \\
\end{tabular}}
\caption{Data for the proof of Corollary \ref{cor:Dusartcor} for $19\leq p\leq 769$.}\label{tab:dataDusartCor1}
\end{table}

\begin{table}[H]
\centerline{\begin{tabular}{lll|lll|lll}
\multicolumn{1}{c}{$p$} & \multicolumn{1}{c}{$f(p)$} & \multicolumn{1}{c|}{$g(p)$} & \multicolumn{1}{c}{$p$} & \multicolumn{1}{c}{$f(p)$} & \multicolumn{1}{c|}{$g(p)$} & \multicolumn{1}{c}{$p$} & \multicolumn{1}{c}{$f(p)$} & \multicolumn{1}{c}{$g(p)$} \\
\hline
773 & 0.084 & 0.078 & 1069 & 0.080 & 0.074 & 1409 & 0.077 & 0.071 \\
787 & 0.084 & 0.078 & 1087 & 0.080 & 0.074 & 1423 & 0.077 & 0.071 \\
797 & 0.084 & 0.078 & 1091 & 0.080 & 0.074 & 1427 & 0.077 & 0.071 \\
809 & 0.084 & 0.078 & 1093 & 0.080 & 0.074 & 1429 & 0.077 & 0.071 \\
811 & 0.084 & 0.078 & 1097 & 0.080 & 0.074 & 1433 & 0.077 & 0.071 \\
821 & 0.083 & 0.077 & 1103 & 0.080 & 0.074 & 1439 & 0.077 & 0.071 \\
823 & 0.083 & 0.077 & 1109 & 0.080 & 0.074 & 1447 & 0.077 & 0.071 \\
827 & 0.083 & 0.077 & 1117 & 0.080 & 0.074 & 1451 & 0.077 & 0.071 \\
829 & 0.083 & 0.077 & 1123 & 0.080 & 0.074 & 1453 & 0.077 & 0.071 \\
839 & 0.083 & 0.077 & 1129 & 0.079 & 0.074 & 1459 & 0.077 & 0.071 \\
853 & 0.083 & 0.077 & 1151 & 0.079 & 0.074 & 1471 & 0.077 & 0.071 \\
857 & 0.083 & 0.077 & 1153 & 0.079 & 0.074 & 1481 & 0.077 & 0.071 \\
859 & 0.083 & 0.077 & 1163 & 0.079 & 0.073 & 1483 & 0.077 & 0.071 \\
863 & 0.083 & 0.077 & 1171 & 0.079 & 0.073 & 1487 & 0.077 & 0.071 \\
877 & 0.083 & 0.077 & 1181 & 0.079 & 0.073 & 1489 & 0.077 & 0.071 \\
881 & 0.082 & 0.077 & 1187 & 0.079 & 0.073 & 1493 & 0.077 & 0.071 \\
883 & 0.082 & 0.077 & 1193 & 0.079 & 0.073 & 1499 & 0.076 & 0.071 \\
887 & 0.082 & 0.077 & 1201 & 0.079 & 0.073 & 1511 & 0.076 & 0.071 \\
907 & 0.082 & 0.076 & 1213 & 0.079 & 0.073 & 1523 & 0.076 & 0.071 \\
911 & 0.082 & 0.076 & 1217 & 0.079 & 0.073 & 1531 & 0.076 & 0.070 \\
919 & 0.082 & 0.076 & 1223 & 0.079 & 0.073 & 1543 & 0.076 & 0.070 \\
929 & 0.082 & 0.076 & 1229 & 0.079 & 0.073 & 1549 & 0.076 & 0.070 \\
937 & 0.082 & 0.076 & 1231 & 0.079 & 0.073 & 1553 & 0.076 & 0.070 \\
941 & 0.082 & 0.076 & 1237 & 0.079 & 0.073 & 1559 & 0.076 & 0.070 \\
947 & 0.082 & 0.076 & 1249 & 0.078 & 0.073 & 1567 & 0.076 & 0.070 \\
953 & 0.082 & 0.076 & 1259 & 0.078 & 0.073 & 1571 & 0.076 & 0.070 \\
967 & 0.081 & 0.076 & 1277 & 0.078 & 0.072 & 1579 & 0.076 & 0.070 \\
971 & 0.081 & 0.075 & 1279 & 0.078 & 0.072 & 1583 & 0.076 & 0.070 \\
977 & 0.081 & 0.075 & 1283 & 0.078 & 0.072 & 1597 & 0.076 & 0.070 \\
983 & 0.081 & 0.075 & 1289 & 0.078 & 0.072 & 1601 & 0.076 & 0.070 \\
991 & 0.081 & 0.075 & 1291 & 0.078 & 0.072 & 1607 & 0.076 & 0.070 \\
997 & 0.081 & 0.075 & 1297 & 0.078 & 0.072 & 1609 & 0.076 & 0.070 \\
1009 & 0.081 & 0.075 & 1301 & 0.078 & 0.072 & 1613 & 0.076 & 0.070 \\
1013 & 0.081 & 0.075 & 1303 & 0.078 & 0.072 & 1619 & 0.076 & 0.070 \\
1019 & 0.081 & 0.075 & 1307 & 0.078 & 0.072 & 1621 & 0.076 & 0.070 \\
1021 & 0.081 & 0.075 & 1319 & 0.078 & 0.072 & 1627 & 0.076 & 0.070 \\
1031 & 0.081 & 0.075 & 1321 & 0.078 & 0.072 & 1637 & 0.076 & 0.070 \\
1033 & 0.081 & 0.075 & 1327 & 0.078 & 0.072 & 1657 & 0.075 & 0.070 \\
1039 & 0.080 & 0.075 & 1361 & 0.078 & 0.072 & 1663 & 0.075 & 0.070 \\
1049 & 0.080 & 0.075 & 1367 & 0.078 & 0.072 & 1667 & 0.075 & 0.070 \\
1051 & 0.080 & 0.075 & 1373 & 0.078 & 0.072 & 1669 & 0.075 & 0.070 \\
1061 & 0.080 & 0.074 & 1381 & 0.077 & 0.072 & 1693 & 0.075 & 0.069 \\
1063 & 0.080 & 0.074 & 1399 & 0.077 & 0.071 & 1697 & 0.075 & 0.069 \\
\end{tabular}}
\caption{Data for the proof of Corollary \ref{cor:Dusartcor} for $773\leq p\leq 1697$.}\label{tab:dataDusartCor2}
\end{table}

\begin{table}[H]
\centerline{\begin{tabular}{lll|lll|lll}
\multicolumn{1}{c}{$p$} & \multicolumn{1}{c}{$f(p)$} & \multicolumn{1}{c|}{$g(p)$} & \multicolumn{1}{c}{$p$} & \multicolumn{1}{c}{$f(p)$} & \multicolumn{1}{c|}{$g(p)$} & \multicolumn{1}{c}{$p$} & \multicolumn{1}{c}{$f(p)$} & \multicolumn{1}{c}{$g(p)$} \\
\hline
1699 & 0.075 & 0.069 & 2039 & 0.073 & 0.068 & 2377 & 0.072 & 0.066 \\
1709 & 0.075 & 0.069 & 2053 & 0.073 & 0.068 & 2381 & 0.072 & 0.066 \\
1721 & 0.075 & 0.069 & 2063 & 0.073 & 0.068 & 2383 & 0.072 & 0.066 \\
1723 & 0.075 & 0.069 & 2069 & 0.073 & 0.067 & 2389 & 0.072 & 0.066 \\
1733 & 0.075 & 0.069 & 2081 & 0.073 & 0.067 & 2393 & 0.072 & 0.066 \\
1741 & 0.075 & 0.069 & 2083 & 0.073 & 0.067 & 2399 & 0.072 & 0.066 \\
1747 & 0.075 & 0.069 & 2087 & 0.073 & 0.067 & 2411 & 0.072 & 0.066 \\
1753 & 0.075 & 0.069 & 2089 & 0.073 & 0.067 & 2417 & 0.072 & 0.066 \\
1759 & 0.075 & 0.069 & 2099 & 0.073 & 0.067 & 2423 & 0.072 & 0.066 \\
1777 & 0.075 & 0.069 & 2111 & 0.073 & 0.067 & 2437 & 0.072 & 0.066 \\
1783 & 0.075 & 0.069 & 2113 & 0.073 & 0.067 & 2441 & 0.072 & 0.066 \\
1787 & 0.075 & 0.069 & 2129 & 0.073 & 0.067 & 2447 & 0.072 & 0.066 \\
1789 & 0.075 & 0.069 & 2131 & 0.073 & 0.067 & 2459 & 0.072 & 0.066 \\
1801 & 0.075 & 0.069 & 2137 & 0.073 & 0.067 & 2467 & 0.072 & 0.066 \\
1811 & 0.075 & 0.069 & 2141 & 0.073 & 0.067 & 2473 & 0.072 & 0.066 \\
1823 & 0.075 & 0.069 & 2143 & 0.073 & 0.067 & 2477 & 0.072 & 0.066 \\
1831 & 0.075 & 0.069 & 2153 & 0.073 & 0.067 & 2503 & 0.072 & 0.066 \\
1847 & 0.074 & 0.069 & 2161 & 0.073 & 0.067 & 2521 & 0.072 & 0.066 \\
1861 & 0.074 & 0.069 & 2179 & 0.073 & 0.067 & 2531 & 0.072 & 0.066 \\
1867 & 0.074 & 0.068 & 2203 & 0.073 & 0.067 & 2539 & 0.072 & 0.066 \\
1871 & 0.074 & 0.068 & 2207 & 0.073 & 0.067 & 2543 & 0.071 & 0.066 \\
1873 & 0.074 & 0.068 & 2213 & 0.073 & 0.067 & 2549 & 0.071 & 0.066 \\
1877 & 0.074 & 0.068 & 2221 & 0.073 & 0.067 & 2551 & 0.071 & 0.066 \\
1879 & 0.074 & 0.068 & 2237 & 0.073 & 0.067 & 2557 & 0.071 & 0.066 \\
1889 & 0.074 & 0.068 & 2239 & 0.073 & 0.067 & 2579 & 0.071 & 0.065 \\
1901 & 0.074 & 0.068 & 2243 & 0.073 & 0.067 & 2591 & 0.071 & 0.065 \\
1907 & 0.074 & 0.068 & 2251 & 0.073 & 0.067 & 2593 & 0.071 & 0.065 \\
1913 & 0.074 & 0.068 & 2267 & 0.073 & 0.067 & 2609 & 0.071 & 0.065 \\
1931 & 0.074 & 0.068 & 2269 & 0.073 & 0.067 & 2617 & 0.071 & 0.065 \\
1933 & 0.074 & 0.068 & 2273 & 0.073 & 0.067 & 2621 & 0.071 & 0.065 \\
1949 & 0.074 & 0.068 & 2281 & 0.072 & 0.067 & 2633 & 0.071 & 0.065 \\
1951 & 0.074 & 0.068 & 2287 & 0.072 & 0.067 & 2647 & 0.071 & 0.065 \\
1973 & 0.074 & 0.068 & 2293 & 0.072 & 0.067 & 2657 & 0.071 & 0.065 \\
1979 & 0.074 & 0.068 & 2297 & 0.072 & 0.067 & 2659 & 0.071 & 0.065 \\
1987 & 0.074 & 0.068 & 2309 & 0.072 & 0.066 & 2663 & 0.071 & 0.065 \\
1993 & 0.074 & 0.068 & 2311 & 0.072 & 0.066 & 2671 & 0.071 & 0.065 \\
1997 & 0.074 & 0.068 & 2333 & 0.072 & 0.066 & 2677 & 0.071 & 0.065 \\
1999 & 0.074 & 0.068 & 2339 & 0.072 & 0.066 & 2683 & 0.071 & 0.065 \\
2003 & 0.074 & 0.068 & 2341 & 0.072 & 0.066 & 2687 & 0.071 & 0.065 \\
2011 & 0.074 & 0.068 & 2347 & 0.072 & 0.066 & 2689 & 0.071 & 0.065 \\
2017 & 0.074 & 0.068 & 2351 & 0.072 & 0.066 & 2693 & 0.071 & 0.065 \\
2027 & 0.074 & 0.068 & 2357 & 0.072 & 0.066 & 2699 & 0.071 & 0.065 \\
2029 & 0.074 & 0.068 & 2371 & 0.072 & 0.066 & 2707 & 0.071 & 0.065 \\
\end{tabular}}
\caption{Data for the proof of Corollary \ref{cor:Dusartcor} for $1699\leq p\leq 2707$.}\label{tab:dataDusartCor3}
\end{table}

\begin{table}[H]
\centerline{\begin{tabular}{lll|lll|lll}
\multicolumn{1}{c}{$p$} & \multicolumn{1}{c}{$f(p)$} & \multicolumn{1}{c|}{$g(p)$} & \multicolumn{1}{c}{$p$} & \multicolumn{1}{c}{$f(p)$} & \multicolumn{1}{c|}{$g(p)$} & \multicolumn{1}{c}{$p$} & \multicolumn{1}{c}{$f(p)$} & \multicolumn{1}{c}{$g(p)$} \\
\hline
2711 & 0.071 & 0.065 & 2797 & 0.071 & 0.065 & 2897 & 0.070 & 0.064 \\
2713 & 0.071 & 0.065 & 2801 & 0.071 & 0.065 & 2903 & 0.070 & 0.064 \\
2719 & 0.071 & 0.065 & 2803 & 0.070 & 0.065 & 2909 & 0.070 & 0.064 \\
2729 & 0.071 & 0.065 & 2819 & 0.070 & 0.065 & 2917 & 0.070 & 0.064 \\
2731 & 0.071 & 0.065 & 2833 & 0.070 & 0.065 & 2927 & 0.070 & 0.064 \\
2741 & 0.071 & 0.065 & 2837 & 0.070 & 0.065 & 2939 & 0.070 & 0.064 \\
2749 & 0.071 & 0.065 & 2843 & 0.070 & 0.065 & 2953 & 0.070 & 0.064 \\
2753 & 0.071 & 0.065 & 2851 & 0.070 & 0.065 & 2957 & 0.070 & 0.064 \\
2767 & 0.071 & 0.065 & 2857 & 0.070 & 0.065 & 2963 & 0.070 & 0.064 \\
2777 & 0.071 & 0.065 & 2861 & 0.070 & 0.065 & 2969 & 0.070 & 0.064 \\
2789 & 0.071 & 0.065 & 2879 & 0.070 & 0.064 & 2971 & 0.070 & 0.064 \\
2791 & 0.071 & 0.065 & 2887 & 0.070 & 0.064 & & & \\
\end{tabular}}
\caption{Data for the proof of Corollary \ref{cor:Dusartcor} for $2711\leq p\leq 2971$.}\label{tab:dataDusartCor4}
\end{table}

\begin{table}[H]
\centerline{\begin{tabular}{lll|lll|lll|lll}
\multicolumn{1}{c}{$n$} & \multicolumn{1}{c}{$\Phi$} & \multicolumn{1}{c|}{$b$} & \multicolumn{1}{c}{$n$} & \multicolumn{1}{c}{$\Phi$} & \multicolumn{1}{c|}{$b$} & \multicolumn{1}{c}{$n$} & \multicolumn{1}{c}{$\Phi$} & \multicolumn{1}{c|}{$b$} & \multicolumn{1}{c}{$n$} & \multicolumn{1}{c}{$\Phi$} & \multicolumn{1}{c}{$b$} \\
\hline
70 & 16 & 13 & 83 & 20 & 16 & 96 & 21 & 18 & 109 & 26 & 20 \\
71 & 17 & 13 & 84 & 20 & 16 & 97 & 22 & 18 & 110 & 26 & 21 \\
72 & 17 & 14 & 85 & 20 & 16 & 98 & 22 & 18 & 111 & 26 & 21 \\
73 & 18 & 14 & 86 & 20 & 16 & 99 & 22 & 19 & 112 & 26 & 21 \\
74 & 18 & 14 & 87 & 20 & 16 & 100 & 22 & 19 & 113 & 27 & 21 \\
75 & 18 & 14 & 88 & 20 & 17 & 101 & 23 & 19 & 114 & 27 & 21 \\
76 & 18 & 14 & 89 & 21 & 17 & 102 & 23 & 19 & 115 & 27 & 21 \\
77 & 18 & 15 & 90 & 21 & 17 & 103 & 24 & 19 & 116 & 27 & 22 \\
78 & 18 & 15 & 91 & 21 & 17 & 104 & 24 & 19 & 117 & 27 & 22 \\
79 & 19 & 15 & 92 & 21 & 17 & 105 & 24 & 20 & 118 & 27 & 22 \\
80 & 19 & 15 & 93 & 21 & 17 & 106 & 24 & 20 & 119 & 27 & 22 \\
81 & 19 & 15 & 94 & 21 & 18 & 107 & 25 & 20 & 120 & 27 & 22 \\
82 & 19 & 15 & 95 & 21 & 18 & 108 & 25 & 20 & 121 & 28 & 23 \\
\end{tabular}}
\caption{Data for the proof of Lemma \ref{lem:phiboundcase4} for $p=11$.}\label{tab:dataphi4a}
\end{table}

\begin{table}[H]
\centerline{\begin{tabular}{lll|lll|lll|lll}
\multicolumn{1}{c}{$n$} & \multicolumn{1}{c}{$\Phi$} & \multicolumn{1}{c|}{$b$} & \multicolumn{1}{c}{$n$} & \multicolumn{1}{c}{$\Phi$} & \multicolumn{1}{c|}{$b$} & \multicolumn{1}{c}{$n$} & \multicolumn{1}{c}{$\Phi$} & \multicolumn{1}{c|}{$b$} & \multicolumn{1}{c}{$n$} & \multicolumn{1}{c}{$\Phi$} & \multicolumn{1}{c}{$b$} \\
\hline
190 & 39 & 30 & 199 & 43 & 31 & 208 & 43 & 33 & 217 & 44 & 34 \\
191 & 40 & 30 & 200 & 43 & 31 & 209 & 43 & 33 & 218 & 44 & 34 \\
192 & 40 & 30 & 201 & 43 & 31 & 210 & 43 & 33 & 219 & 44 & 34 \\
193 & 41 & 30 & 202 & 43 & 32 & 211 & 44 & 33 & 220 & 44 & 34 \\
194 & 41 & 30 & 203 & 43 & 32 & 212 & 44 & 33 & 221 & 45 & 35 \\
195 & 41 & 31 & 204 & 43 & 32 & 213 & 44 & 33 & 222 & 45 & 35 \\
196 & 41 & 31 & 205 & 43 & 32 & 214 & 44 & 33 & 223 & 46 & 35 \\
197 & 42 & 31 & 206 & 43 & 32 & 215 & 44 & 34 & & & \\
198 & 42 & 31 & 207 & 43 & 32 & 216 & 44 & 34 & & & \\
\end{tabular}}
\caption{Data for the proof of Lemma \ref{lem:phiboundcase4} for $p=13$.}\label{tab:dataphi4b}
\end{table}


\end{document}